\documentclass[final,12pt,a4paper]{amsart}

 \usepackage[matrix,arrow,curve]{xy}
 \usepackage{graphicx}
 \usepackage[english]{babel}
 \usepackage{amsmath}
 \usepackage{amssymb, indentfirst}
 \usepackage{amsthm}
 \usepackage{enumitem}
 \usepackage{soul}
 \usepackage{verbatim}
 \usepackage{tikz}
 \usepackage{xypic}
 \usepackage{fullpage}
 \usepackage{ulem}
 
\usepackage{color}

\newcommand{\kk}{{\bf k}} %k - the field
\newcommand{\Db}[1]{D^b(#1)} % bounded derived category

\newcommand{\rmod}{{\rm mod} \text{-}} % category of modules
\newcommand{\smod}{{\rm\underline{mod} \text{-}}} % stable category
 
\newcommand{\soc}{{\rm soc}} % socle
\newcommand{\rk}{{\rm rk}} % rank
\newcommand{\rad}{{\rm rad}} % radical
\newcommand{\im}{{\rm Im}} % image

\newtheorem{cor}{Corollary}[section]
\newtheorem{lem}[cor]{Lemma}
\newtheorem{rem}[cor]{Remark}
\newtheorem{ex}[cor]{Example}
\newtheorem{opr}[cor]{Definition}
\newtheorem{thm}[cor]{Theorem}
\newtheorem{prop}[cor]{Proposition}

\begin{document}
\title{Brauer graph algebras are closed under derived equivalence}
\author{Mikhail Antipov$^1$}
\thanks{$^1$Higher School of Economics, Saint Petersburg, Russia} 
\author{Alexandra Zvonareva$^2$}
\thanks{$^2$Institute of Algebra and Number Theory, University of Stuttgart, Germany, alexandra.zvonareva@mathematik.uni-stuttgart.de}
\date{}

\maketitle
\vspace{-20pt}
\begin{abstract}
In this paper  the  class  of  Brauer  graph  algebras  is proved to be closed  under  derived  equivalence. For that we use the rank of the maximal torus  of the identity component $Out^0(A)$ of the group of outer automorphisms of a symmetric stably biserial algebra $A$. 
\end{abstract}

\medskip


\medskip


\section{Introduction}
Brauer graph algebras or equivalently symmetric special biserial algebras, originating from modular representation theory,  are  studied quite extensively.  They appear in classifications of various classes of algebras including blocks with cyclic or dihedral defect groups \cite{Dade,Don}, blocks of Hecke algebras \cite{Ar1,Ar2} and others. Brauer tree algebras, the subclass of Brauer graph algebras of finite representation type, contain all blocks with cyclic defect group. 

In this paper we make a final step in the proof of the fact that Brauer graph algebras  are closed under derived equivalence.  This fact was believed to be true, based on the work 
of Pogorza\l{}y \cite{Pog}. In \cite{AIP},
counterexamples to some of the statements of \cite{Pog} were given. In \cite{AZ1}, we revised the proof of the fact that the only algebras possibly stably (and thus derived) equivalent to self-injective special biserial algebras (a class containing Brauer graph algebras) are self-injective stably biserial (see Section \ref{SecPre}). As a finite-dimensional algebra  derived equivalent  to a  symmetric algebra  is itself symmetric \cite{RicDe}, we can restrict our attention to symmetric stably biserial algebras. It turns out that in odd characteristic the class of symmetric stably biserial algebras coincides with the class of Brauer graph algebras, whereas in characteristic 2 this is not the case \cite{AZ1}.

The general strategy of the proof of the fact that Brauer graph algebras are closed under derived equivalence follows the classical proof for Brauer tree algebras. The fact that Brauer tree algebras are closed under stable equivalence was proved in \cite{GabRie}. Since by \cite{RicSt} derived equivalence for self-injective algebras implies stable equivalence, it follows that this class is closed under derived equivalence as well. It turns out that the proof for the whole class of Brauer graph algebras is much more involved and requires an extra step in characteristic 2, which is provided in this paper. 

A symmetric stably biserial algebra can be given by the same combinatorial data as the Brauer graph algebra, that is a graph on a surface and a number attached to each vertex of this graph, called the multiplicity. Additionally, one needs to fix a distinguished class of loops in the quiver, satisfying certain conditions,  which we call deformed loops (see Section \ref{SecPre}). In case the number of deformed loops is 0 we recover the usual definition of a Brauer graph algebra. Since for local algebras derived equivalence implies Morita equivalence \cite{RZ}, we will sometimes assume that $A$ has at least 2 simple modules. For further reference, let us denote by $V(\Gamma), E(\Gamma)$ and $F(\Gamma)$ the vertices, edges and faces of the Brauer graph $\Gamma$.

The main technique, used in this paper, is the computation of the rank of the maximal torus $T(A)$ of the identity component $Out^0(A)$ of the group of outer automorphisms for a symmetric stably biserial algebra $A$. The group $Out^0(A)$  is a derived invariant \cite{HZS,Rouq} used quite seldom. The only previous systematic application we know of is the proof of the fact that the number of arrows in the quiver of a gentle algebra is a derived invariant \cite{AAG}.
 
\begin{thm}\label{TheoremRank} 
Let $\kk$ be an algebraically closed field.  Let $A$ be a symmetric stably biserial algebra over $\kk$ ($char(\kk)=2$) or  a symmetric special biserial algebra over $\kk$ ($char(\kk)\neq 2$) with at least two non-isomorphic simple modules, which is not a caterpillar (see Section \ref{NNN}). Let $\Gamma$ be the Brauer graph of $A$ and let $d$ be the number of deformed loops in $A$ ($d=0$ for the symmetric special biserial case). The rank of $T(A)$ is $|E(\Gamma)|-|V(\Gamma)|-d+2$.
\end{thm}

In section \ref{SectionComb}, we revisit the known derived invariants for Brauer graph algebras \cite{AntStab, AntDer, AntTh, AntStabGr} in arbitrary characteristic, providing simpler proofs of their invariance for the larger class of symmetric stably biserial algebras and correcting some inaccuracies in the existing literature.

\begin{thm}\label{TheoremDerInv}
Let $A$ be a symmetric stably biserial algebra with a Brauer graph $\Gamma$ and with at least two simple modules. The following are invariants of $A$ under a derived equivalence of symmetric stably biserial algebras: $|V(\Gamma)|$, $|E(\Gamma)|$, $|F(\Gamma)|$, the multiset of perimeters of faces, the multiset of multiplicities, and bipartivity of $\Gamma$.
\end{thm}

As a corollary of Theorems \ref{TheoremRank} and \ref{TheoremDerInv} and the fact that Brauer graph algebras can be derived equivalent only to symmetric stably biserial algebras \cite{AZ1} we obtain the following:

\begin{cor}\label{CorClosed}
The class of Brauer graph algebras is closed under derived equivalence. Namely, if $A$ is an algebra Morita equivalent to a Brauer graph algebra and $B$ is an algebra such that $\Db A \simeq \Db B$, then $B$ is  Morita equivalent to a Brauer graph algebra.
\end{cor}

In forthcoming work \cite{Gnedin}, among other results, W. Gnedin independently obtains
Corollary \ref{CorClosed} in characteristic 2 and for bipartite Brauer graphs by different methods. Note that the list of invariants from Theorem \ref{TheoremDerInv} is crucial to the forthcoming joint work \cite{OpperZvo} of S. Opper and the second named author, where a complete  classification of Brauer graph algebras up to derived equivalence will be provided.

\vspace{10pt}
\textbf{Acknowledgement:} AZ would like to thank Alexey Ananyevskiy for many fruitful discussions.

\section{Preliminaries}\label{SecPre}

Throughout this paper, $A$ is a basic, connected, finite dimensional
algebra over an algebraically closed field $\kk$ and $\rmod A$ is the
category of finite dimensional right $A$-modules. The stable category of $\rmod A$ will be denoted by $\smod A$ and $\Omega:
\smod A \rightarrow \smod A$ will denote the syzygy functor. The bounded derived category of the category $\rmod A$ will be denoted by $\Db A$. A quiver $Q$ consists of a set of vertices $Q_0$ and a set of arrows $Q_1$. The map $s: Q_1 \rightarrow Q_0$ will denote the start of an arrow, the map $e: Q_1 \rightarrow Q_0$ will denote the end of an arrow. In the path algebra $\kk Q$ the multiplication of arrows $\alpha$ and $\beta$ is $\alpha\beta\neq 0$, provided $e(\alpha)=s(\beta)$, by $e_v$ we will denote the primitive idempotent corresponding to the vertex $v\in Q_0$, by $J(A)$ we will denote the Jacobson radical of the algebra $A$, which is the ideal generated by the arrows of the quiver $Q$ in case $A\simeq \kk Q/I$. By $K_0(\mathcal{C})$ we are going to denote the Grothendieck group of an Abelian or a triangulated category $\mathcal{C}$. 

In this paper we are going to be interested in symmetric special biserial and symmetric stably biserial algebras.

\begin{opr}\label{ssb} Let $Q$ be a quiver, $I$ an admissible ideal of $\kk Q$. A
self-injective  algebra $A = \kk Q/I$ is called \textbf{special biserial} if the
following conditions are satisfied.
\begin{enumerate}
    \item For each vertex $v \in Q$, the number of outgoing arrows and the
number of incoming arrows are less than or equal to 2.
\item  For each arrow $\alpha \in Q$, there is at most one arrow $\beta
\in Q$ such that $\alpha \beta \notin I$.
\item  For each arrow $\alpha \in Q$, there is at most one arrow $\beta
\in Q$ such that $\beta\alpha  \notin I$.
\end{enumerate}
\end{opr}

\begin{opr}\label{stb} Let $Q$ be a quiver, $I$ an admissible ideal of $\kk Q$. A
self-injective  algebra $A = \kk Q/I$ is called \textbf{stably biserial} if the
following conditions are satisfied.
\begin{enumerate}
    \item For each vertex $v \in Q$, the number of outgoing arrows and the
number of incoming arrows are less than or equal to 2.
    \item For each arrow $\alpha \in Q$, there is at most one arrow $\beta
\in Q$ such that $\alpha \beta \not\in \alpha \rad (A) \beta +
\soc (A)$.
    \item For each arrow $\alpha \in Q$, there is at most one arrow $\beta
\in Q$ such that $\beta\alpha  \not\in  \beta\rad (A) \alpha +
\soc (A)$.
\end{enumerate}
\end{opr}

The following description of stably biserial algebras was provided
in \cite{AIP}:

\begin{prop}[Proposition 7.5 \cite{AIP}]\label{PropAIP} If $A = \kk Q/I$ is stably biserial then we can
choose a presentation of $A$ in such a way that the following
conditions hold.
\begin{enumerate}
    \item If $\alpha\beta \neq 0$, $\alpha\gamma \neq 0$, $\beta \neq
\gamma$, for arrows $\alpha, \beta, \gamma$ then either $\alpha\beta
\in \soc (A)$ or $\alpha\gamma \in \soc (A)$.
    \item If $\beta\alpha \neq 0$, $\gamma\alpha \neq 0$, $\beta \neq
\gamma$, for arrows $\alpha, \beta, \gamma$ then either $\beta\alpha
\in \soc (A)$ or $\gamma\alpha \in \soc (A)$.
\end{enumerate}
\end{prop}

Self-injective special biserial algebras are a subclass of stably biserial algebras. We will call an algebra \textbf{symmetric special biserial} (SSB for short) or \textbf{symmetric stably biserial}, if in addition to being special biserial or stably biserial it is symmetric. 

Consider the following set-up:

\begin{enumerate}
\item A quiver $Q$ such that every vertex has two incoming and two outgoing arrows.
\item A permutation $\pi$ on $Q_1$ with $e(\alpha)=s(\pi(\alpha))$ for all $\alpha\in Q_1$.
\item A function $m:C(\pi)\to \mathbb{N}$, where $C(\pi)$ is the set of cycles of $\pi$ and $\mathbb{N}$ denotes the set of natural numbers without zero. We will denote by $C(\alpha):=\alpha\pi(\alpha)\pi^2(\alpha)\dots\pi^{|C( \alpha)|-1}$ the cycle, containing $\alpha\in Q_1$ and call $m(C(\alpha))$  the \textbf{multiplicity} of the cycle $C(\alpha)$.
\item A set $\mathcal{L}=\{\alpha_{i_1},\dots,\alpha_{i_d}\}$ of loops, such that $\pi(\alpha_{i_j})\neq \alpha_{i_j}$ and a set of elements $\{t_{\alpha_{i_1}},\dots,t_{\alpha_{i_d}}\},$ with $t_{\alpha_{i_j}} \in \kk^*$.
\end{enumerate}

In \cite{AZ1} the following description of symmetric stably biserial algebras in terms of generators and relations was obtained:

\begin{thm}\label{TheoremDescriptionOfStablyBi}
Any symmetric stably biserial algebra has a presentation $A=\kk Q/I$, where $Q$ is a quiver as in the set-up above and the ideal of relations $I$ is generated by

\begin{enumerate}
\item  $\alpha\beta$ for all $\alpha,\beta \in Q_1$,
$\beta\neq\pi(\alpha)$, $\alpha\notin\mathcal{L}$,
\item $C(\alpha)^{m(C(\alpha))}-C(\beta)^{m(C(\beta))}$ for all $\alpha,\beta\in Q_1$ with $s(\alpha)=s(\beta)$,
\item $\alpha^2-t_{\alpha}C(\alpha)^{m(C(\alpha))}$ for each $\alpha\in \mathcal{L}$,
\item $C(\alpha)^{m(C(\alpha))}\beta$ for all $\alpha,\beta\in Q_1$.
\end{enumerate}

\noindent Moreover, any symmetric stably biserial algebra over an algebraically
closed field $\kk$ with $char(\kk)\neq 2$ is isomorphic to an algebra $\kk Q/I$ as above with   $\mathcal{L}=\emptyset$.
\end{thm}

\begin{rem}
Note that the relations of the form (2) appearing in Theorem \ref{TheoremDescriptionOfStablyBi} are not admissible in the case where there is a cycle of $\pi$ consisting of one arrow and having multiplicity $1$. In \cite{AZ1} we considered quivers $Q$ such that every vertex has either two incoming and two outgoing arrows or one incoming and one outgoing arrow, and an admissible ideal of relations $I$. To pass to this equivalent description from the description in Theorem \ref{TheoremDescriptionOfStablyBi} one needs to delete loops $\alpha$ such that $\pi(\alpha)=\alpha,$ $m(\alpha)=1$ and modify the ideal of relations accordingly. In the case where the algebra has only two loops $\alpha_i$ such that $\pi(\alpha_i)=\alpha_i,$ $m(\alpha_i)=1$, one needs to delete only one loop to get the algebra isomorphic to $\kk[x]/(x^2)$. 

Finally, note that in case $\mathcal{L}=\emptyset$, the algebras appearing in Theorem \ref{TheoremDescriptionOfStablyBi} are symmetric special biserial independent of the characteristic of the field by construction.
\end{rem}

The loops from the set $\mathcal{L}$ will be called \textbf{deformed loops}. It is well known that any SSB-algebra can be given in the above form with the empty set of deformed loops.

Note also that for the description of stably biserial algebras from Theorem \ref{TheoremDescriptionOfStablyBi}  the number of arrows 
$|Q_1|=2|Q_0|$ is invariant under derived equivalence, since $|Q_0|$ is the rank of  $K_0(\Db A)$.

By \cite{Rog, AntGen, Schr1} the class of SSB-algebras coincide with the class of Brauer graph algebras. \textbf{Brauer graph} is a graph with a cyclic ordering of (half-)edges around each vertex and a number assigned to each vertex. This graph $\Gamma$ can be constructed using the data $(Q, \pi, m)$ as follows: the vertices of $\Gamma$ correspond to the cycles of $\pi$, the edges of $\Gamma$ correspond to the vertices of $Q$, an edge connects two vertices of $\Gamma$ if the corresponding $\pi$-cycles have the corresponding vertex of $Q$ in common. The cyclic ordering of edges around a vertex comes from the order in which vertices of $Q$ appear in the $\pi$-cycle, the multiplicities come from the function $m$. Along the same lines, to each Brauer graph one can assign the data $(Q,\pi,m)$ and the corresponding SSB-algebra.

In \cite{AntDer} each Brauer graph was considered together with a minimal compact oriented surface $\mathcal{S}$, into which it is embedded, in such a way that its complement is a union of disks (see also \cite{MaSchr}). The ordering of the edges around the vertices of $\Gamma$ comes from the orientation of the surface (we will assume that the edges are ordered clockwise if the graph $\Gamma$ is drawn on the plane, representing part of a sphere). Now it makes sense to consider not only vertices and edges of $\Gamma$ but also \textbf{faces} of $\Gamma$.
Using the edges of the graph $\Gamma$ the surface $\mathcal{S}$ can be cut into polygons, by a \textbf{perimeter} of a face $F$ we mean the number of edges in the corresponding polygon, thus, for example, the perimeter of a self-folded triangle is 3.
The set  $\mathcal{L}$ corresponds to a subset of faces of perimeter 1 of $\Gamma$.  We will use the terms SSB-algebra and Brauer graph algebra interchangeably. 

From the discussion above it follows that to any symmetric stably biserial algebra given as a path algebra of a quiver with relations as in Theorem \ref{TheoremDescriptionOfStablyBi} we can associate a Brauer graph $\Gamma$ together with a subset of its faces of perimeter $1$. We will prove in Lemma \ref{LemBrauer} below that, with one exception, $\Gamma$ does not depend on the choice of the presentation of the algebra.

\begin{ex}\label{ExRef}
Let us consider the Brauer graph $\Gamma$ and the corresponding quiver $Q$
\begin{center}
\begin{tikzpicture}
\node (v1) at (-4,1) {$1$};
\node (v2) at (-1.5,1) {$2$};

\draw[<-]  (v1) edge[bend right=30] (v2);

\draw[<-]  (v2) edge[bend right=30] (v1);

\node at (-3.1,1.1) {$\alpha$};
\node at (-2.35,0.85) {$\beta$};
\node at (-0.05,1) {$\gamma$};
\node at (-4.5,1) {$Q:$};

\node (v5) at (-10.1,1.3) {$1$};
\node (v6) at (-7.5,1) {$2$};

\node (v4) at (-11.85,1) {$\Gamma:$};
\node (v3) at (-8.85,1) {$\cdot$};
\node (v4) at (-11.35,1) {$\cdot$};
\draw[-]  (v3) edge (v4);
\draw[-]  plot[smooth, tension=.7] coordinates {(-8.7,0.85) (-7.85,0.7) (-7.85,1.3) (-8.7,1.15)};

\draw[<-]  plot[smooth, tension=.7] coordinates {(-1.2,0.85) (-0.35,0.7) (-0.35,1.3) (-1.2,1.15)};

\end{tikzpicture}
\end{center}
the multiplicities of both vertices are assumed to be $1$. By construction we get $\pi(\alpha)=\gamma,$ $\pi(\gamma)=\beta,$ and $\pi(\beta)=\alpha$. To that Brauer graph we can associate a Brauer graph algebra $A=\kk Q/I$, where $I$ is the ideal generated by relations $$\alpha\beta=0, \gamma^2=0, \beta\alpha\gamma=\gamma\beta\alpha \text{ and } \alpha\gamma\beta\alpha=0.$$ On the other hand, we can consider the symmetric stably biserial algebra $A_{def}=\kk Q/I'$, where $I'$ is the ideal generated by relations $$\alpha\beta=0, \gamma^2=\beta\alpha\gamma=\gamma\beta\alpha \text{ and } \alpha\gamma\beta\alpha=0.$$ This corresponds to deforming the loop $\gamma$, associated to the face of $\Gamma$ of perimeter $1$, and taking $t_\gamma=1$. The fact that the algebra $A_{def}$ is indeed symmetric stably biserial will follow from Proposition \ref{PropSym}.

In case, $char(\kk)\neq 2$ one can use the following change of basis $\alpha'=\alpha,$ $\beta'=\beta$ and  $\gamma'=\gamma-\frac{\beta\alpha}{2}$ of the algebra $A_{def}$ to see that $A_{def}\simeq A$. In case, $char(\kk)= 2$ the algebras $A_{def}$ and $A$ are not isomorphic and are not even derived equivalent, which will become apparent from Theorem \ref{TheoremRank}. 
\end{ex}

\section{Stably biserial algebras}\label{NNN}

In this section we are going to investigate basic properties of symmetric stably biserial algebras. As stated in Theorem \ref{TheoremDescriptionOfStablyBi}, any symmetric stably biserial algebra $A$ can be given as a certain deformation of a Brauer graph algebra with a Brauer graph $\Gamma$. We are going to show that with one exception the Brauer graph $\Gamma$ does not depend on the presentation of $A$ and that any deformation from Theorem \ref{TheoremDescriptionOfStablyBi} is indeed symmetric.

Let us introduce a special class of algebras, called caterpillar in this paper. This class of algebras behaves differently from other symmetric special biserial algebras and has to be excluded from some considerations.

The algebra $\kk Q/I$ will be called a \textbf{caterpillar} of length $n> 1$ if $Q$ is of the form 
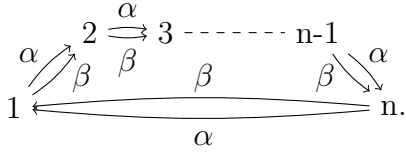
\begin{center}
\begin{tikzpicture}

\node (v1) at (-3,0) {1};
\node (v2) at (-2,1) {2};
\node (v3) at (-1,1) {3};
\node (v4) at (1,1) {n-1};
\node (v5) at (2,0) {n.};

\draw[->]  (v1) edge[bend right=10] (v2);
\draw[->]  (v1) edge[bend left=10] (v2);

\draw[->]  (v2) edge[bend right=10] (v3);
\draw[->]  (v2) edge[bend left=10] (v3);

\draw  (v3) edge[dashed] (v4);

\draw[->]  (v4) edge[bend right=10] (v5);
\draw[->]  (v4) edge[bend left=10] (v5);

\draw[->]  (v5) edge[bend right=5] (v1);
\draw[->]  (v5) edge[bend left=5] (v1);
\node at (-2.8,0.7) {$\alpha$};
\node at (-1.5,1.3) {$\alpha$};
\node at (1.8,0.7) {$\alpha$};
\node at (-0.5,-0.4) {$\alpha$};
\node at (-2.1,0.4) {$\beta$};
\node at (-1.5,0.6) {$\beta$};
\node at (1.1,0.4) {$\beta$};
\node at (-0.5,0.4) {$\beta$};
\end{tikzpicture}
\end{center}
In this case, the ideal of relations can be either of the form $I_1$ or of the form $I_2$. The ideal $I_1$ is generated by relations $\alpha e_i \beta=0=\beta e_i\alpha,$ for  $i\neq 1$, $\alpha e_1 \alpha=0=\beta e_1\beta$, $(\alpha^{k}e_1 \beta^n \alpha^{n-k}) ^{m_{\alpha}}=(\beta^{k}e_1 \alpha^n \beta^{n-k}) ^{m_{\alpha}}$, thus there is one $\pi$-cycle with multiplicity $m_{\alpha}$. The ideal $I_2$ is generated by relations $\alpha\beta=0=\beta\alpha,$ $\alpha^{nm_{\alpha}}=\beta^{nm_{\beta}}$, thus there are two $\pi$-cycles with multiplicities $m_{\alpha}$ and $m_{\beta}$. The Brauer graphs of these algebras are:
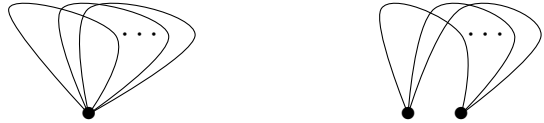
\begin{center}
\begin{tikzpicture}[scale=0.7]

\node (v1) at (0.5,-0.5) {$\bullet$};
\draw  plot[smooth, tension=.7] coordinates {(v1) (2.5,1) (0.5,1.5) (0.5,-0.5)};
\draw  plot[smooth, tension=.7] coordinates {(0.5,-0.5) (2,1) (0,1.5) (0.5,-0.5)};
\node at (1.5,1) {$\cdots$};
\draw  plot[smooth, tension=.7] coordinates {(0.5,-0.5) (1,1) (-1,1.5) (0.5,-0.5)};
\node (v3) at (6.5,-0.5) {$\bullet$};
\node (v2) at (7.5,-0.5) {$\bullet$};
\draw  plot[smooth, tension=.7] coordinates {(v2) (9,1) (7.5,1.5) (v3)};
\draw  plot[smooth, tension=.7] coordinates {(7.5,-0.5) (8.5,1) (7,1.5) (6.5,-0.5)};
\node at (8,1) {$\cdots$};
\draw  plot[smooth, tension=.7] coordinates {(7.5,-0.5) (7.5,1) (6,1.5) (6.5,-0.5)};
\end{tikzpicture}
\end{center}

The following Lemma is most likely known for Brauer graph algebras, but we could not find a proof in the literature, so we include it for the larger class of symmetric stably biserial algebras.

\begin{lem}\label{LemBrauer}
Let $A$ be a symmetric stably biserial algebra with a presentation $\kk Q/I$ as in Theorem \ref{TheoremDescriptionOfStablyBi} and the associated Brauer graph $\Gamma$. If $\Gamma$ is not a loop with 1 as the multiplicity of the unique vertex, or an edge with 2 as the multiplicity of both vertices, then $\Gamma$ does not depend on the choice of the presentation $\kk Q/I$.  
\end{lem}

\begin{proof}
Let $A$ be a symmetric stably biserial algebra with a presentation $\kk Q/I$ as in Theorem \ref{TheoremDescriptionOfStablyBi} and the associated Brauer graph $\Gamma$. The structure of projective modules over $A$ can be easily deduced from the Brauer graph $\Gamma$ of $\kk Q/I$. We will show that the structure of the Brauer graph $\Gamma$ can be deduced from the structure of projective modules in the category $\rmod A$ and thus does not depend on the choice of the presentation $\kk Q/I$.

Assume that the algebra $A$ has two presentations $\kk Q/I \simeq \kk Q'/I'$ as in Theorem \ref{TheoremDescriptionOfStablyBi}. Let us delete loops  $\alpha$ such that $\pi(\alpha)=\alpha,$ $m(\alpha)=1$ and modify the ideal of relations accordingly in both presentations, we still have an isomorphism between the modified presentations, this isomorphism induces an equivalence between the categories $\rmod \kk Q/I$ and $\rmod \kk Q'/I'$. The deleted loops correspond to the leaves with multiplicity 1 in the Brauer graph and can always be reconstructed from the valency of the vertices in the quiver. Since the ideals of relations $I$ and $I'$ are admissible after the deletion of extra loops, we can assume that $Q$ and $Q'$ coincide, thus there is a bijection between primitive idempotents for these two presentations and between simple modules over $\kk Q/I$ and $\kk Q'/I'$. This bijection between the simple modules coincides with the bijection induces by the equivalence of the categories of modules. Simple modules identified under the bijection will be denoted by $S_i$. This bijection extends to a bijection between the edges of the Brauer graphs $\Gamma$ and $\Gamma'$, constructed from these two presentations, since the edges of the Brauer graph correspond to simple modules.

Next we will reconstruct the Brauer graphs $\Gamma$ and $\Gamma'$ from the module categories $\rmod \kk Q/I$ and $\rmod \kk Q'/I'$ and deduce that the graphs coincide, since the categories are equivalent. Let us consider $\kk Q/I$, $\kk Q'/I'$ is analogous. For the projective cover $P_i$ of $S_i$ we can consider the module $\rad P_i/\soc P_i$ which has either one or two indecomposable summands $M_i$ and $N_i$. These modules are uniserial and each of them gives a unique sequence of simple modules, corresponding to it's radical series $(S_{i_1}, \cdots,S_{i_n})$, where $S_{i_1}$ is the top of $M_i$ or $N_i$ respectively.  Adding $S_i$ to this sequence $(S_{i_0}=S_i, S_{i_1}, \cdots,S_{i_n})$ and numbering the sequence by the elements of $\mathbb{Z}/(n+1)\mathbb{Z}$ we get a collection of cycles of simple modules (coming from each $P_i$ for all $i$'s), which we identify up to a cyclic permutation of $\mathbb{Z}/(n+1)\mathbb{Z}$. If for some $P_i$ the modules $M_i$ and $N_i$ are both zero, then $A\simeq \kk[x]/(x^2)$. The radical series of the modules $M_i$, $N_i$ do not depend on the presentation of the algebra, so in this case the Brauer graph is determined uniquely and is an edge with both vertices of multiplicity 1.

Note that by construction of the permutation $\pi$ the cyclic ordering of the simples in the sequences constructed above coincides with the cyclic ordering of edges in the Brauer graph. 

If the module $S_i$ appears in two different cyclic sequences, then the edge, corresponding to $S_i$ is not a loop and we can reconstruct the cyclic ordering around the ends of the edge, corresponding to $S_i$ from the subsequence of the form $(S_i, S_{i_1}, \cdots,S_{i_l}, S_i)$, where $(S_{i_1}, \cdots,S_{i_l})$ does not contain $S_i$. The multiplicities of the vertices is the number of times the subsequences $(S_i, S_{i_1}, \cdots,S_{i_l})$ has to be repeated to get the whole sequences. 

If $S_i$ appears in only one cyclic sequence, but this cyclic sequence has a subsequence of the form $\sigma=(S_i, S_{i_1}, \cdots,S_{i_l}, S_i,S_{i_{l+2}}, \cdots,S_{i_m}, S_i)$, where the subsequences $(S_{i_1}, \cdots,S_{i_l})$ and $(S_{i_{l+2}}, \cdots,S_{i_m})$ do not contain $S_i$, are different and at least one of them is not empty, then the edge corresponding to $S_i$ is a loop and we can reconstruct the cyclic ordering of the edges around the vertex adjacent to this loop and the multiplicity is the number of times the subsequence $(S_i, S_{i_1}, \cdots,S_{i_l}, S_i,S_{i_{l+2}}, \cdots,S_{i_m})$ has to be repeated to get the whole sequence.

If $S_i$ appears in only one cyclic sequence and this sequence does not have a subsequence of the form $\sigma$, and the projective module $P_i$ is uniserial then we can reconstruct the cyclic ordering of the edges around one vertex incident to the edge corresponding to $S_i$ and its multiplicity as before, the other end of this edge has no other edges incident to it and has multiplicity 1.

The only case left to consider is when $S_i$ appears in only one cyclic sequence and this sequence does not have a subsequence of the form $\sigma$, but the projective module $P_i$ is not uniserial. In this case the modules $M_i$ and $N_i$ have the same radical series but are both nonzero. If the cyclic sequence containing $S_i$ is of the form $(S_i, S_{i_1}, \cdots,S_{i_l}, S_i)$, where $(S_{i_1}, \cdots,S_{i_l})$ does not contain $S_i$, then the edge, corresponding to $S_i$ is not a loop and we can reconstruct the cyclic ordering around each end of this edge, the multiplicities of the ends are 1. Assume that  $(S_i, S_{i_1}, \cdots,S_{i_l}, S_i)$, where $(S_{i_1}, \cdots,S_{i_l})$ does not contain $S_i$ is a subsequence of the cyclic sequence and it has to be repeated $m>1$ times to get the whole sequence. If $(S_{i_1}, \cdots,S_{i_l})$ is empty, then $|Q_0|=1$, this situation will be considered later. If the edge corresponding to $S_i$ is a loop, then all edges corresponding to $(S_{i_1}, \cdots,S_{i_l})$ are loops and we get a caterpillar with one vertex in the Brauer graph with multiplicity $m/2$ (this can happen only for even $m$). If the edge corresponding to $S_i$ is not a loop, then all edges corresponding to $(S_{i_1}, \cdots,S_{i_l})$ are not loops and we get a caterpillar with two vertices in the Brauer graph, both with multiplicity $m$. The two algebras we get for even $m$ are not isomorphic, since they are not even derived equivalent by Proposition \ref{cycles}. Note that the proof of Proposition \ref{cycles} does not rely on the results of this section.

Let us consider the case $|Q_0|=1$. The Brauer graph is either an edge or a loop. If it is an edge, there are no deformed loops and $A\simeq A_{k,l}=\kk[x,y]/\langle xy, x^k-y^l\rangle$, $k,l\geq 1$, which is a commutative algebra. If it is a loop, then for multiplicity greater than one, $A$ is non-commutative. So it is sufficient to consider the algebra 
$B_{t_x,t_y}=\kk[x,y]/\langle x^2y,y^2x, x^2-t_x xy, y^2-t_y xy \rangle$, which is $4$-dimensional. If it is isomorphic to $A_{k,l}$, then either $k=1, l=3$, which is not possible, or $k=l=2$. In the last case the algebras can, indeed, be isomorphic, even when $t_x=t_y=0$, $char(\kk)\neq 2$. 
\end{proof}

\begin{rem}
The cyclic ordering of edges in the Brauer graph played an important role in the proof of Theorem \ref{TheoremDescriptionOfStablyBi}. Namely, for a symmetric stably biserial algebra $A$ with an arbitrary presentation as in Proposition \ref{PropAIP}, with an admissible ideal of relations, we first fixed the permutation $\pi$ and then using the change of basis produced a presentation as in Theorem \ref{TheoremDescriptionOfStablyBi}. We would like to note here that the change of basis from \cite[Lemma 10]{AZ1} does not work for the algebras $A_t$ and $B_{t,s}$ (see below), which was not noted in the proof of Lemma 10. This does not effect the result, since these algebras turn out not to be symmetric. For the algebra $A_t$ the element $\alpha-t\beta$ belongs to the socle of $A_t$, for the algebra $B_{t,s}$ the element $\gamma_0-s\gamma_1$ belongs to the socle of $B_{t,s}$, which is not possible for symmetric algebras. Here 
\end{rem}

\begin{center}
    \begin{tikzpicture}
\node (v1) at (-4,1) {1};
\node (v2) at (-1.5,1) {2};

\draw[->]  (v1) edge[bend right=30] (v2);
\draw[->]  (v1) edge[bend right=40] (v2);

\draw[->]  (v2) edge[bend right=30] (v1);
\draw[->]  (v2) edge[bend right=40] (v1);
\node at (-3.2,1.75) {$\beta_1$};
\node at (-3.2,1.1) {$\beta_0$};
\node at (-2.25,0.85) {$\gamma_0$};
\node at (-2.25,0.3) {$\gamma_1$};
\node at (-4.5,1) {$Q:$};
\node at (-3,-0.5) {$I=\langle J(\kk Q)^3 , \gamma_0\beta_0-\gamma_1\beta_1, $};

\node at (-3,-1) {$\beta_0\gamma_0-\beta_1\gamma_1, \beta_0\gamma_1-t\beta_0\gamma_0, $};

\node at (-3,-1.5) {$\gamma_1\beta_0-t\gamma_1\beta_1, \beta_1\gamma_0-s\beta_1\gamma_1, $};

\node at (-3,-2) {$\gamma_0\beta_1-s\gamma_0\beta_0 \rangle, st=1$};

\node at (-3,-2.5) {$B_{t,s}=\kk Q/I, \pi(\gamma_i)=\beta_i, \pi(\beta_i)=\gamma_i$};
\node at (-10.35,1) {$Q:$};

\node (v3) at (-8.85,1) {1};

\draw[->]  plot[smooth, tension=.7] coordinates {(-8.7,0.85) (-7.85,0.7) (-7.85,1.3) (-8.7,1.15)};
\draw[->]  plot[smooth, tension=.7] coordinates {(-9,0.85) (-9.85,0.7) (-9.85,1.3) (-9,1.15)};
\node at (-9.1,-0.5) {$I=\langle J(\kk Q)^3 , \alpha^2-\beta^2, $};
\node at (-9.1,-1) {$\alpha\beta- t\alpha^2, \beta\alpha-t\beta^2\rangle, t^2=1$};
\node at (-9.1,-1.5) {$A_{t}=\kk Q/I, \pi(\alpha)=\beta, \pi(\beta)=\alpha$};
\node at (-9.35,1.5) {$\alpha$};
\node at (-8.35,1.5) {$\beta$};
\end{tikzpicture}
\end{center}

Let us denote by $A_\infty$ an algebra isomorphic to any symmetric stably biserial algebras with one vertex, two loops and one $\pi$-cycle of multiplicity $1$.

\begin{prop}\label{PropSym}
Let us consider any data of the form $(Q,\pi,m,\mathcal{L},\{t_\alpha\}_{\alpha\in\mathcal{L}})$, and $A\simeq \kk Q/I$, where $I$ is the ideal of relations described in Theorem \ref{TheoremDescriptionOfStablyBi}. If $A$ is not isomorphic to $A_\infty$ with both loops deformed, then the algebra $A$ is symmetric. 
\end{prop}

\begin{proof}
Recall that an algebra $A$ is symmetric if and only if there exists a non-degenerate symmetric
$\kk$-bilinear form $\langle a, b \rangle: A \times A \rightarrow \kk$ such that $\langle ab, c \rangle = \langle a, bc \rangle$ for all $a, b, c \in A$.

Let us define the standard bilinear form $\langle a, b \rangle:=\phi(ab)$, where the value of $\phi$ on the path basis of $A$ is defined as follows: $\phi(C(\alpha)^{m(C(\alpha))})=1$ for any arrow $\alpha$, thus $\phi(\gamma^2)=t_\gamma$ for any deformed loop $\gamma$, and $\phi(p)=0$ for any path $p\not\in \soc A$. The values of $\phi$ on other elements of $A$ is defined by linearity.

The defined form is bilinear, symmetric and satisfies the property $\langle ab, c \rangle = \langle a, bc \rangle$ for all $a, b, c \in A$. Let us check that it is non-degenerate.

Let us assume that $\phi$ is degenerate, that is $\phi((\sum c_i p_i)a)=0$ for some $\sum c_i p_i\neq 0$ and for all $a\in A$, where $c_i\in \kk^*$ and $p_i$ are paths, by the symmetry of $\phi$, for all $a\in A$ we have $\phi(a(\sum c_i p_i))=0$. We can assume that all $p_i$ start at the same vertex $i$ and end at the same vertex $j$ (multiplying by two idempotents and keeping $\sum c_i p_i$ non-zero). All $p_i$'s  can be written as subpaths of the standard socle paths of the form $C(\alpha)^{m(C(\alpha))}$, we will consider only this presentation of $p_i$'s. Since there are at most two such standard socle paths starting at $i$, all $p_i$'s can be divided into two groups depending on the socle path. Let us choose the shortest path from one of these two groups $p_1$ (by the shortest path we mean the path containing the least number of arrows). Let $\bar{p_1}p_1 $ be the standard socle path containing $p_1$  (that is not $\gamma^2$ for a deformed loop $\gamma$), then $\bar{p_1}(\sum c_i p_i)=c_1\bar{p_1}p_1 +c_2\bar{p_1}p_2$, where $\bar{p_1}p_2$ is non-zero only in case when the shortest path from the second group is an arrow $p_2$ and $\bar{p_1}$ is also an arrow. In this case $p_2$ must be a deformed loop $p_2=\bar{p_1}$. If $\bar{p_1}p_2=0$, then $\phi(\bar{p_1}(\sum c_i p_i))=0$ iff $c_1=0$ and we are done.

Let us do the same exchanging $p_1$ and $p_2$. Then $\bar{p_2}(\sum c_i p_i)=c_2\bar{p_2}p_2 +c_1\bar{p_2}p_1$, where $\bar{p_2}p_1$ appears only in case when the shortest path from the first group is an arrow $p_1$ and $\bar{p_2}$ is also an arrow. In this case $p_1$ must be a deformed loop $p_1=\bar{p_2}$. Thus we get exactly the excluded case of 2 deformed loops at one vertex.
\end{proof}

\section{Combinatorial derived invariants}\label{SectionComb}

The aim of this section is to show that the following combinatorial data are invariant under derived equivalences of stably biserial algebras: number of vertices, edges and faces of the Brauer graph, multisets of perimeters of faces, multisets of multiplicities of vertices, bipartivity. Note that the corresponding results were shown to be true  for Brauer graph algebras with some minor inaccuracies in \cite{AntGroth, AntStab, AntDer, AntStabGr}, the proofs are identical or rely on the corresponding results for Brauer graph algebras, except for some simplifications.  From here on we are going
to exclude the case $|Q_0|=1$ from some considerations, since by \cite{RZ} a local algebra
can be derived equivalent only to itself.

\subsection{The centre of a symmetric stably biserial algebra} 

In this subsection we compute the centre $Z(A)$ of a symmetric stably biserial algebra $A$, which is known to be invariant under derived equivalence, see  \cite{Ric}.   We will use this to establish, that the number of $\pi$-cycles, or the vertices of the Brauer graph, is invariant under derived equivalence. This will also gives us an opportunity to correct the above-mentioned inaccuracies in the description of the centre of an SSB-algebra made in \cite{AntDer}. Let $\{C_1, C_2, \dots ,
C_r\}$ be the set of $\pi$-cycles. For each $i = 1, \dots , r$
consider a cyclic sequence $(\alpha_{i,1}, \alpha_{i,2}, \dots ,
\alpha_{i,l_i})$  of arrows of the cycle $C_i$, where $\pi(\alpha_{i,j})=\alpha_{i,j+1}$ and $l_i$ denotes the length of the cycle $C_i$. Let $m(C_1), m(C_2),
\dots , m(C_r)$ denote the multiplicities of the $\pi$-cycles and let $r' \leq r$ be an integer such that $m(C_i) > 1, i=0,\dots, r'$ and $m(C_i)=1, i=r'+1, \dots, r$. For each
loop $\gamma$ such that $\pi(\gamma) \neq \gamma$ there are $i$ and $j$ such that $\gamma=\alpha_{i,j}$. For each such loop $\gamma$ set
 $q_{\gamma}=q_{\alpha_{i,j}} = (\alpha_{i,j+1}\alpha_{i,j+2} \dots
 \alpha_{i,l_i} \alpha_{i,1} \dots\alpha_{i,j})^{m(C_i)-1}\alpha_{i,j+1}\alpha_{i,j+2} \dots  \alpha_{i,l_i} \alpha_{i,1} \dots\alpha_{i,j-1}$. 

%The following proposition is an analogue of [Proposition 2.1]\cite{AntD}:

\begin{prop}\label{center}
Let $A$ be a symmetric stably biserial algebra corresponding to the data $(Q, \pi, m, \mathcal{L})$ and let $\Gamma$ be its Brauer graph. As a vector space over $\kk$ the centre $Z(A)$ is generated by $1$ and
by the elements of the following form:

a) Elements $m_{i,t} = (\alpha_{i,1}\alpha_{i,2} \dots
\alpha_{i,l_i})^t +(\alpha_{i,2}\alpha_{i,3} \dots
\alpha_{i,1})^t + \dots +(\alpha_{i,l_i}\alpha_{i,1} \dots
\alpha_{i,l_i-1})^t$, for $i = 1, 2,\dots , r'$ and $t = 1, \dots,
m(C_i)-1$.

b) Elements $q_{\gamma}$ for each
loop $\gamma$ such that $\pi(\gamma) \neq \gamma$.

c) Elements $s_v$ for each vertex $v \in Q_0$, where $s_v$ is a
socle element corresponding to the vertex $v$.

Moreover, if $A$ is not isomorphic to some $A_\infty$, then considered as an algebra, $Z(A)/(\soc (Z(A)) \simeq \kk[x_1, x_2, \dots , x_{r'}]/\langle
x_i^{m(C_i)}, (x_ix_j)_{i\neq j}\rangle$. So the multiset $\{m(C_1), m(C_2), \dots, m(C_{r'})\}$ is
invariant under derived equivalence. The number of loops $\gamma$ such that $\pi(\gamma) \neq \gamma$, or equivalently, the number of faces of $\Gamma$ of perimeter 1 is a derived invariant as well.
\end{prop}

\begin{proof}
It is clear that all the listed elements belong to the centre of
$A$. Let us prove that any element in the centre is a linear
combination of elements of the form (a)-(c).

Each $z \in Z(A)$ has the form $z =
\sum_{i=1}^Na_ip_i+z'$, where $p_i$ are the elements of the path basis of $A$ which do not belong to
the socle of $A$ and $z'\in \soc (A)$. Without loss of generality, we can assume $z'=0$. All elements $p_i$ with $a_i\neq 0$ are necessarily closed paths, that is $p_i=e_vp_ie_v$ for some idempotent $e_v$, corresponding to a vertex $v$. Fix $p_i=\beta_1\beta_2 \dots
\beta_m$ with $\beta_j\in Q_1$, let $\beta_{m+1}=\pi(\beta_m)$,
then $p_i\beta_{m+1} \neq 0$. Assume that $p_i\beta_{m+1}$ does
not belong to the socle of $A$, then $\beta_1\beta_2 \dots \beta_m\beta_{m+1}$ has coefficient $a_i$ in the sum
$\beta_{m+1}z$, hence $\beta_{m+1}=\beta_{1}$ and the coefficient
of $\beta_2 \dots \beta_m\beta_{m+1}$ in $z$ is $a_i$, so
$p_i=(\alpha_{j,s}\alpha_{j,s+1}\dots \alpha_{j,s-1})^t$
for some $\pi$-cycle, and $z$ contains $a_im_{j,t}$ as a summand,
$z-a_im_{j,t}$ contains less summands, then $z$. If
$\beta_1\beta_2 \dots \beta_m\beta_{m+1}$ belongs to
the socle of $A$, then  $\beta_{m+1}$ is a loop, since  $p_i$ is a closed path. Then $p_i$ is either
$m_{j,m(C_j)-1}$ for a cycle $C_j$, consisting of a single loop (if $\pi(\beta_{m+1})=\beta_{m+1}$) or
$q_{\beta_{m+1}}$ (if $\pi(\beta_{m+1})\neq\beta_{m+1}$). Either $z-a_iq_{\beta_{m+1}}$ or
$z-a_im_{j,m(C_j)-1}$ has less summands then $z$ and we can proceed by
induction on the number of nonzero coefficients $a_i$ in the sum $z =
\sum_{i=1}^Na_ip_i$. By induction we get that $z$ is a linear combination of elements of the form (a)-(c).

In case $A\not\simeq A_\infty$, $\soc( Z(A))$ is generated by the elements of type (b) and (c).  Moreover, $m_{i,t_1}
m_{j,t_2} = \delta_{i,j}m_{i,t_1+t_2}$ and $m^{m(C_i)}_{i,1} \in
\soc Z(A)$. Hence $Z(A)/(\soc (Z(A))) \simeq \kk[x_1, x_2,\dots,
x_r']/\langle x_i^{m(C_i)}, (x_ix_j)_{i\neq j}\rangle$. Since $Z(A)$ is
invariant under derived equivalence as an algebra, the multiset $\{m(C_1),
m(C_2), . . . , m(C_{r'})\}$ is invariant under derived equivalence. The socle of $Z(A)$ is spanned by the elements of the form $s_v$, $v\in Q_0$ and $q_\gamma$ for loops $\gamma$ such that $\pi(\gamma)\neq\gamma$. Since the number of the elements $s_v$ is a derived invariant, so is the number of loops $\gamma$ such that $\pi(\gamma)\neq\gamma$.
\end{proof}

\begin{cor}\label{CatLem}
Let $A_1,A_2$ be derived equivalent symmetric stably biserial algebras, where $A_1$ is a caterpillar. Then $A_2$ is special biserial.
\end{cor}

\begin{proof}
The algebra $A_1$ has no loops $\gamma$ such that $\pi(\gamma)\neq \gamma$, so, by Proposition \ref{center}, $A_2$ has no such loops as well, hence, $A_2$ is symmetric special biserial.
\end{proof}

\subsection{Number and perimeters of faces}

Let  $A$ be a symmetric stably biserial algebra with the corresponding Brauer graph $\Gamma$, let $p_1, p_2,\dots ,p_m$ be the perimeters of faces of $\Gamma$. 
Note that the perimeter of a face $F$ coincides  with the length of the corresponding Green walk (see \cite{Rog, Duf}). The aim of this section is to prove that the multiset $\{p_1,\dots ,p_m\}$ (and, in particular, the number $m$ of faces of  $\Gamma$) is an invariant of the derived category of  $A$. For this we are going to use the structure of the Auslander-Reiten quiver of the stable category of $\rmod A$. Note that by \cite{RicSt} for self-injective and in particular for symmetric algebras derived equivalence implies stable equivalence, so any invariant of  stable equivalence is automatically a derived invariant.

Indecomposable modules over special biserial algebras are classified in terms of strings and bands, the description of the Auslander-Reiten sequences and of the Auslander-Reiten quiver for such algebras is well understood \cite{GelfandPonomarev, ButlerRingel, WW, Erd, ErdSk}. Let us consider the $AR$-quiver $\Gamma_{\smod A}$ of $\smod A$. If $A$ is SSB, then each periodic component of $\Gamma_{\smod A}$ is a tube. Moreover, all tubes are either tubes of rank 1, consisting of band modules, or tubes consisting of string modules, the latter tubes are called exceptional. Exceptional tubes correspond to faces  of $\Gamma$: if a face has an even perimeter $p$, then it produces two tubes of rank $p/2$, which are permuted by $\Omega$; if a face has an odd perimeter, then it produces one tube of rank $p$, which is stable under the action of $\Omega$. For a detailed exposition see \cite[Section 4]{Duf}. For a face $F$ let us denote by  $\mathcal{M}_F$ the set of modules in the mouth of the tube or tubes, corresponding to $F$. They can be constructed as follows: take two consecutive edges $v$ and $w$ from the face $F$, they correspond to two vertices $v$ and $w$ in the quiver of $A$ and an arrow $v\xrightarrow{\alpha}w$; the modules in $\mathcal{M}_F$ are of the form $P_v/\alpha P_w$. Thus, the modules which appear in the mouths of exceptional tubes are simple modules whose projective covers
are uniserial and all maximal uniserial quotients of indecomposable projective modules, which are not projective. Let us denote the set of these modules by $\mathcal{M}$. This set of modules is stable under $\Omega$.

In case $A$ is symmetric stably biserial algebra with the data $(Q,\pi,m,\mathcal{L})$, the $AR$-quiver $\Gamma_{\smod A}$ of $\smod A$ coincides with the same quiver for the SSB-algebra $A'$ constructed from the  data $(Q,\pi,m)$. Indeed, $A/\soc(A)\simeq A'/\soc (A')$ is a string algebra, so the classification of indecomposable non-projective modules is the same. The AR-sequences not ending at the module $P/\soc (P)$ for a projective module $P$ coincide for $A$ and $A/\soc(A)$ by \cite[Proposition 4.5]{AusRei}, the fact that the sequences $0\rightarrow \rad P\rightarrow \rad P/\soc P\oplus P\rightarrow P/\soc P\rightarrow 0$ give the same in $\Gamma_{\smod A}$ and $\Gamma_{\smod A'}$ can be checked by hand. 

The set of modules $\mathcal{M}$ in the mouths of exceptional tubes for $A$ and $A'$ is the same under the identification from the previous paragraph. 
Let us compare the behaviour of these modules under the action of $\Omega$.
 The modules of the form $P_v/\gamma P_v$ for a deformed loop $\gamma$ at a vertex $v$ are now sent by $\Omega$ to band modules. All other modules from $\mathcal{M}$ are permuted by $\Omega$ in the same manner for $A$ and $A'$.  If there is an exceptional tube of rank at least 2 (corresponding to a face of perimeter at least 3), then the modules in its mouth are not of the form $P_v/\gamma P_v$ for a deformed loop $\gamma$ (such modules belong to tubes of rank 1), hence, the modules from $\mathcal{M}$ which lie in the mouths of the tubes of rank at least 2 are permuted by $\Omega$ in the same manner for $A$ and $A'$.  This means that $\Omega$ permutes exceptional tubes of rank at least 2 for $A'$ in the same manner as for $A$.

Consequently, we get that the number of faces of a given perimeter $p>2$ is the number of tubes of rank $p$, stable under $\Omega$, in case $p$ is odd and the number of tubes of rank $p/2$, not stable under $\Omega$, divided by $2$, in case $p$ is even for both $A$ and $A'$. In both cases the perimeter can also be reconstructed from the stable category. By Proposition \ref{center} the number of faces of  perimeter 1 is a derived invariant. The number of faces of perimeter $2$ can be reconstructed as follows: $(2|E(\Gamma)|-\sum_{p_i\neq 2}p_i)/2$. Thus, the following holds:

\begin{prop}\label{faces}
Let $A_1$, $A_2$ be two symmetric stably biserial algebras with Brauer graphs $\Gamma_1$ and $\Gamma_2$, such that neither $\Gamma_1$ nor $\Gamma_2$ is a loop with multiplicity 1 or an edge with multiplicity of both vertices 2. If $\Db{A_1}\simeq \Db{A_2}$, then the number of faces and the multisets of perimeters of  faces of $\Gamma_1$ and $\Gamma_2$ coincide.
\end{prop}

\begin{rem}
The proof of the fact that  the number of faces and the multisets of perimeters of  faces of $\Gamma$ is invariant under an equivalence of stable categories of SSB-algebras was provided in \cite{AntStab} with a mistake, which was corrected in \cite{AntTh}. Note that the proof is much more involved, since one can not use the centre of the algebra $Z(A)$ (as in Proposition \ref{center}), so one has to deal with the tubes coming from the faces of perimeter 1 and 2 and with the tubes containing band modules. 
\end{rem}

\subsection{Number of vertices and their multiplicities}

Let $\mathbb{Z}^{|Q_0|}$ be the Grothendieck group of a self-injective algebra $A$ with the Cartan matrix $C(A)$. Then $C(A)$ defines a group homomorphism $\phi_A$ from $\mathbb{Z}^{|Q_0|}$ to itself and $K_0(\smod A)\simeq \mathbb{Z}^{|Q_0|}/\im (\phi_A)$. To obtain the standard description of this Abelian group one can use Smith's normal form of $C(A)$, which can be obtained by computing the greatest common divisors of all $t\times t$ minors of $C(A)$, this was done for SSB-algebras in \cite{AntGroth, AntStabGr}. Let $Q_1/\pi$ be the set of orbits of $Q_1$ under the action of $\pi$. We are going to use only the rank of $C(A)$, which is equal to  $|Q_1/\pi|-1$  if the Brauer graph $\Gamma$ of $A$ is bipartite and to $|Q_1/\pi|$, otherwise. Note that by construction $|Q_1/\pi|$ is the number of vertices of $\Gamma$.

\begin{prop}\label{cycles}
Let $A_1$, $A_2$ be   stably biserial algebras with Brauer graphs $\Gamma_1$ and $\Gamma_2$,  such that neither $\Gamma_1$ nor $\Gamma_2$ is a loop with multiplicity 1 or an edge with multiplicity 2 at both vertices. If $\Db {A_1} \simeq \Db {A_2}$, then $|V(\Gamma_1)|=|V(\Gamma_2)|$. Moreover, the multisets of multiplicities of the vertices and the bipartivity of $\Gamma_1$ and $\Gamma_2$ coincide.
\end{prop}

\begin{proof}
Let $A_i'$ be the special biserial algebras corresponding to the data given by the Brauer graphs $\Gamma_i$, for $i=1,2$.  As $A_i$ and $A_i'$ have the same Cartan matrices, we can use the description of the structure of the Grothendieck group of $A_i'$ for $A_i$. Since derived equivalences of self-injective algebras imply stable equivalences, $K_0(\smod A_1)\simeq K_0(\smod A_2)$, thus the ranks of $C(A_1)$ and $C(A_2)$ coincide.

By \cite{AntStabGr}, $\rk (C(A_i))$ is equal to $|V(\Gamma_i)|-1$  if the Brauer graph $\Gamma(A_i)$ of $A_i$ is bipartite and to $|V(\Gamma_i)|$ 
 otherwise. By Proposition \ref{faces} $|F(\Gamma_1)|=|F(\Gamma_2)|$, the same holds for $|E(\Gamma_1)|=|E(\Gamma_2)|$. Since $|V(\Gamma_i)|-|E(\Gamma_i)|+|F(\Gamma_i)|$ is even, as the Euler characteristic of the corresponding surface $\mathcal{S}_i$, we see that $|V(\Gamma_i)|$'s can not differ by 1, hence, $|V(\Gamma_1)|=|V(\Gamma_2)|$. Since the ranks of the Cartan matrices of $A_1$ and $A_2$ coincide, $\Gamma_i$ are either simultaneously bipartite or simultaneously not bipartite.
 
 The multiplicities of  vertices $m(C_i)>1$ can be detected by the centre of the algebra, see Proposition \ref{center}. The invariance of the number of vertices with multiplicity 1 follows from the invariance of the number of all vertices.
\end{proof}

\begin{proof}[Proof of Theorem \ref{TheoremDerInv}]
Combining the results of Proposition \ref{faces} and \ref{cycles} we get that the  following  are  derived  invariants  of a symmetric stably biserial algebra $A$ with at least two non-isomorphic simple modules and the corresponding Brauer graph $\Gamma$: $|V(\Gamma)|,|E(\Gamma)|,|F(\Gamma)|$, the multiset of perimeters of faces, the multiset of multiplicities of vertices, bipartivity of $\Gamma$.
\end{proof}

\begin{ex} Let $A$ and $A_{def}$ be the algebras considered in Exampla \ref{ExRef}.
Recall that the Brauer graph corresponding to both $A$ and $A_{def}$ is
\begin{center}
\begin{tikzpicture}
\node (v4) at (-11.85,1) {$\Gamma:$};
\node (v3) at (-8.85,1) {$\cdot$};
\node (v4) at (-11.35,1) {$\cdot$};
\draw[-]  (v3) edge (v4);
\draw[-]  plot[smooth, tension=.7] coordinates {(-8.7,0.85) (-7.85,0.7) (-7.85,1.3) (-8.7,1.15)};

\end{tikzpicture},
\end{center}

\noindent where $\Gamma$ is embedded into a sphere. Theorem \ref{TheoremDerInv} provides the following list of derided invariants for $A$ and  $A_{def}$, which all coincide:
$|V(\Gamma)|=2$, $|E(\Gamma)|=2$, $|F(\Gamma)|=2$,  the  multiset  of perimeters of faces of $\Gamma$ is $\{1,3\}$, the multiset of multiplicities of vertices of $\Gamma$ is $\{1,1\}$, $\Gamma$ is not bipartite.
\end{ex}

\begin{rem}
We consider algebras over an algebraically closed field in  Theorem~\ref{TheoremDerInv}, simply because its proof relies on the existence of a presentation of any symmetric stably biserial algebra in the form from Theorem \ref{TheoremDescriptionOfStablyBi}, which was shown only for algebraically closed fields. If we restrict only to derived equivalences between Brauer graph algebras, with at least two non-isomorphic simple modules, then  $|V(\Gamma)|,|E(\Gamma)|,|F(\Gamma)|$, the multiset of perimeters of faces, the multiset of multiplicities of vertices and the bipartivity of $\Gamma$ are still derived invariants, where $\Gamma$ as usual denotes the corresponding Brauer graph. Indeed, the proof of Theorem \ref{TheoremDerInv} uses only the computation of the center of the algebra, the description and behaviour of tubes of rank at least $2$ in the stable category (these tubes do not contain any band modules, which are all contained in tubes of rank $1$) and the rank of the Grothendieck group of the stable category. All these invariants do not depend on the assumption that the field is algebraically closed.  Note that this assumption is essential in the next section.
\end{rem}

\section{The group of outer automorphisms}

Throughout this section we are going to assume that either $char(\kk)=2$ or that $char(\kk)\neq2$ and the number of deformed loops is zero, i.e. $d=0$ ($A$ is symmetric special biserial); additionally we are going to assume that $A$ is  not a caterpillar and that $A$ has at least two non-isomorphic simple modules.  We are going to show that derived equivalent symmetric stably biserial algebras have the same number of deformed loops using the identity component $Out^0(A)$ of the group of outer automorphisms. By \cite{HZS,Rouq}, the group  $Out^0(A)$  is
invariant under derived equivalence as an algebraic group, for a finite dimensional algebra $A$ over an algebraically closed field. We are going to use the necessary notions and facts about algebraic groups freely, for more details see \cite{Milne}.

\subsection{$H'$ is trigonalizable}
Let $A=\kk Q/I$ be a symmetric stably biserial algebra in the standard form given in Theorem \ref{TheoremDescriptionOfStablyBi}, i.e. the ideal of relations is not necessarily admissible.
Let $\mathcal{L} \subset Q_1 $ be the set of
deformed loops.   Let $A=B\oplus J(A)$ be
a Wedderburn-Malcev decomposition of $A$ with a  semisimple
subalgebra $B$. It is known that $Out(A)=H/H\cap Inn(A)$, where
$H=\{f\in Aut(A)| f(B)\subset B \}$ \cite{GAS, 
Pol}. If $\{e_v\}_{v\in Q_0}$ is
a set of primitive idempotents corresponding to the vertices of $Q$ and $B=\langle \{e_v\}_{v\in Q_0} \rangle$, then for any $v\in Q_0$ and $f \in H$,
$f(e_v)=e_{v'}$ for some $v'\in Q_0$. Therefore $H'=\{f\in H\mid
f|_B=Id\}$ is a closed subgroup of finite index in $H$, i.e. it is a
union of connected components of $H$. Since $H\cap Inn(A)=H'\cap Inn(A)$ acts on each
component of $H$, the identity component of $Out(A)=H/H\cap Inn(A)$ and of $H'/H'\cap Inn(A)$ coincide. So, without loss of generality, we can consider $H'/H'\cap Inn(A)$ instead of $Out(A)$.

\begin{lem}

If $A$ is not a caterpillar and $\rk K_0(A)\geq 2$, then there is an embedding $i:H' \rightarrow
T(l,\kk)$ of algebraic groups over $\kk$, where $T(l,\kk)$ is the group of lower
triangular matrices and $l=\dim A$.
\end{lem}

\begin{proof}
Let $P$ be a set of paths in $Q$ which forms a basis for $\kk Q/I$,
such that $\alpha^2\notin P$ for $\alpha\in \mathcal{L}$ and all primitive idempotents $\{e_v\}_{v\in Q_0}$ and all arrows of $Q$ are in $P$.
For each $p\in P$ let $l_p=max\{k:p\in J(A)^k\}$, with the convention that $l_{e_v}=0, v\in Q_0$ (so $l_p$ is the length of the longest path equal to $p$). A pair
$(\beta,\beta')\in Q_1\times Q_1$ with $s(\beta)=s(\beta'),
e(\beta)=e(\beta')$ ($\beta , \beta'$ are parallel arrows) and
$\pi^2(\beta)\neq \beta,\pi^2 (\beta')=\beta'$ will be called an exceptional pair.

Let us consider some linear extension of the following partial order on $P$:

1)If $l_p<l_q$, then $p<q$.

2)If $(\beta,\beta')$ is an exceptional pair, then $\beta<\beta'$. Note that since $|Q_0|>1$, $l_{\beta}=l_{\beta'}=1$. 

We are going to express the matrix of an automorphisms of $A$ in the basis $P$  with respect to this linear order and show that, it is lower triangular, i.e, for $f\in H'$ and $p\in P$ we
have $f(p)=k_p p+\sum_{p'>p} k_{p,p'}p'$.

Let us consider $p=\beta \in Q_1,$ such that $l_\beta=1$. If $\beta$ has no
arrows parallel to it, then 
$f(\beta)=k_{\beta}\beta+r$ where $r\in J(A)^2$ and
we are done. Now suppose that $\beta$ and $\beta'$ are two parallel arrows,
in this case we can have $f(\beta)=k_{\beta}\beta+k_{\beta,\beta'}\beta'+r$
with $k_{\beta,\beta'} \neq 0$, $r\in J(A)^2$. Note that since $|Q_0| \neq 1$, $\beta$ is not a loop. There are three possible cases:

1) $\pi(\beta), \pi(\beta')$ are not parallel. In this case $f(\pi(\beta'))=k_{\pi(\beta')}\pi(\beta')+r_1$, $r_1\in J(A)^2$. Then $0=f(\beta\pi(\beta'))=f(\beta)f(\pi(\beta'))=k_{\beta,\beta'}k_{\pi(\beta')}\beta'\pi(\beta') + r',$ $r' \in {J(A)^3}$. A path of length two $\beta'\pi(\beta')$ belongs to $J(A)^3$, hence $\beta'\pi(\beta')\in\rm{soc}(A)$. Since  $\pi(\beta), \pi(\beta')$ are not parallel, $(\beta,\beta')$ is an exceptional pair and $\beta<\beta'$. The same argument for $(\beta',\beta)$ gives $k_{\beta',\beta}k_{\pi(\beta)}\beta\pi(\beta) \in \rm{soc}(A)$, hence $k_{\beta',\beta}k_{\pi(\beta)}=0$, thus  $k_{\beta',\beta}=0$, so $f(\beta')=k_{\beta'}\beta'+r''$, $r''\in J(A)^2$. In particular, this means that $k_{\beta}\in \kk^*$, since $\beta$ belongs to the image of $f$.

2)  $\pi(\beta), \pi(\beta')$ are parallel arrows but $\pi^l(\beta), \pi^l(\beta')$ are not parallel for some $l$ (we take the minimal $l$). In this case $(\pi^{l-1}(\beta),\pi^{l-1}(\beta'))$  is not an exceptional pair (otherwise $s(\pi^{l-1}(\beta))$ has $3$ incoming arrows) 
and $f(\pi^l(\beta'))=k_{\pi^l(\beta')}\pi^l(\beta') + r',$ $r' \in {J(A)^2}$. So $0=f(\pi^{l-1}(\beta))f(\pi^l(\beta'))$ implies $f(\pi^{l-1}(\beta))=k_{\pi^{l-1}(\beta)}\pi^{l-1}(\beta) +r,$ $r\in {J(A)^2}$ and the same holds for $\beta'$. Then by decreasing induction on $i$ we obtain in the same way that
$f(\pi^i(\beta'))=k_{\pi^i(\beta')}\pi^i(\beta') + r',$ $r'\in {J(A)^2}$, $f(\pi^i(\beta))=k_{\pi^i(\beta)}\pi^i(\beta) + r,$ $r\in {J(A)^2}$ for all
$0\leq i\leq l$, in particular, for $i=0$.

3) $\pi^l(\beta),\pi^l(\beta')$ are parallel for all $l$. In this case $A=\kk Q/I$ is a caterpillar.

For an arbitrary
$p_i \in P$, choose a presentation $p_i=\beta_1\dots\beta_n$ with $n$ maximal. Then $f(p_i)=f(\beta_1)\dots
f(\beta_n)=\prod_{i} k_{\beta_i}\beta_1\dots\beta_n +\sum k_jp_j$, where
$p_j>\beta_1\dots\beta_n$. Indeed, $\beta_1\dots\beta_n$ is of the form $\beta_1\pi(\beta_1)\dots\pi^n(\beta_1)$, and since for an exceptional pair $(\beta,\beta')$ we have $\beta'\pi(\beta)=0,\pi^{-1}(\beta)\beta'=0$, the sum $\sum k_jp_j$ belongs to $J(A)^{n+1}$.
\end{proof}

\subsection{Decomposition with the unipotent subgroup} 
Computation of the groups $Out^0(A)$ for all symmetric stably biserial algebras  might turn out to be quite technical, so we want to use some invariant of $Out^0(A)$ preserved by isomorphisms of algebraic groups, which we would be able to compute quite easily. This invariant is the rank of the maximal torus of $Out^0(A)$.

Let us consider maximal unipotent subgroups in $H'$ and $I=H'\cap Inn(A)$, denoted respectively by $U_{H'}$ and $U_I$. These groups are given by the intersection of $H'$ (respectively $I$) with $U(l,\kk)$ the group of (lower) unitriangular matrices. We can consider the following diagram of algebraic groups:

$$
\xymatrix  { 1 \ar[r] & U_I \ar[r]\ar[d] &I\ar[r]\ar[d]& D_I \ar[r]\ar[d] &1\\
1 \ar[r]& U_{H'} \ar[r]\ar[d]& H'\ar[r]\ar[d]& D_{H'} \ar[r]\ar[d] &1\\
1 \ar[r] &U_{H'}/U_I \ar[r]& H'/I\ar[r]& D_{H'}/D_I \ar[r]& 1
}
$$
An easy diagram chasing shows that the map $D_I\rightarrow D_{H'}$ is an embedding. As a quotient of a trigonalizable group $H'/I$ is trigonalizable, $D_{H'}/D_I$ is diagonalizable and $U_{H'}/U_I$ is the maximal unipotent subgroup of $H'/I$, it contains all unipotent subgroups of $H'/I$ \cite[Theorem 16.6]{Milne}. Thus we can consider another diagram:

$$
\xymatrix  {   & 1\ar[d]\ar[r] &1\ar[r]\ar[d]& X \ar[d] &\\
1 \ar[r]& (U_{H'}/U_I)^0 \ar[r]\ar[d]& (H'/I)^0\ar[r]\ar[d]& T(A) \ar[r]\ar[d] &1\\
1 \ar[r] &U_{H'}/U_I \ar[r]\ar[d]& H'/I\ar[r]\ar[d]& D_{H'}/D_I \ar[r]\ar[d]& 1\\
 &Y \ar[r]& Z\ar[r]& W \ar[r]& 1
}
$$

Since $(H'/I)^0$ is connected and solvable its maximal unipotent subgroup is connected and thus coincides with $(U_{H'}/U_I)^0$. Note also that all maximal tori of $(H'/I)^0$ are conjugate. The groups $D_I,D_{H'},D_{H'}/D_I$ are diagonalizable. We also get the following exact sequence $1\rightarrow X\rightarrow Y \rightarrow Z \rightarrow W \rightarrow 1$, where $Y, Z$ are finite, then $X,W$ are finite as well. Hence the rank of $T(A)$ and $D_{H'}/D_I$ coincide. We have proved the following lemma.

\begin{lem}
 In the notation of the previous construction, the rank of $D_{H'}/D_I$ coincides with the rank of the maximal torus $T(A)$ of $Out^0(A)=(H'/I)^0$ and so is a derived invariant of $A$.
\end{lem}

\subsection{Computation of the rank of $D_{H'}/D_I$} Both $D_{H'}$ and $D_I$ are induced by the projection map from the group of lower triangular matrices to the group of diagonal matrices. So we are going to find out what elements can appear on the diagonal of the matrices from $H'$ and $I$. Clearly the diagonal entries corresponding to the arrows of the quiver determine all other diagonal entries, so we are going to restrict our attention to them.  Let $\kk^*$ and $A^*$ be the multiplicative groups of $\kk$ and $A$ respectively.

\begin{lem}
Let $A=\kk Q/I$ be a symmetric stably biserial algebra. There is an isomorphism of affine  algebraic groups $D_I\simeq (\kk^*)^{|Q_0|-1}$.
\end{lem}

\begin{proof} Recall that $Inn(A)=\{f_a| a\in A^* \}$, where $f_a(x)=axa^{-1}$. Each $a\in A^* $ can be uniquely written as
$a=\sum_{i\in Q_0}a_ie_i+r$, where $a_i\in \kk^*$, $r\in J(A)$. Then
$a^{-1}=\sum_{i\in Q_0}a_i^{-1}e_i+r'$, $r'\in J(A)$ and the action of
$f_a$  on $J(A)/J(A)^2$ depends only on $a_i$'s. For $c\in \kk^*$, $f_a$ clearly coincides with $f_{ca}$.

Let $f_a$ be an element of $ H' \cap Inn(A)$ and let  $\overline {f_a}:=f_a\pmod {U_I}$ be the class of $f_a$. Fix any spanning tree $\Delta$ of $Q$ (ignoring the orientation of the arrows of $Q$), let $\{\alpha_i\}_{1\leq i\leq Q_0-1}$ be the corresponding set of arrows, that is $\alpha_i\in \Delta$. For each arrow $\alpha_i$ we have $f_a(\alpha_i)=a\alpha_i a^{-1}=a_{s(\alpha_i)}a^{-1}_{e(\alpha_i)}\alpha_i \pmod {J(A)^2}$. Let us define the map $D_I \xrightarrow{\eta} (\kk^*)^{|Q_0|-1}$ as $\overline{f_a} \to(a_{s(\alpha_1)}a^{-1}_{e(\alpha_1)},\dots, a_{s(\alpha_{{|Q_0|-1}})}a^{-1}_{e(\alpha_{{|Q_0|-1}})})$. The equality $f_a(\alpha_i)=\alpha_i \pmod {J(A)^2}$  for the elements of $U_I$ guarantees that the map is well defined.  

Conversely, consider $(k_1, \dots, k_{|Q_0|-1}) \in (\kk^*)^{|Q_0|-1}$.
Since $\alpha_i$'s form a spanning tree, one can find $\{a_i\}_{i\in Q_0}$ such that $k_i=a_{s(\alpha_1)}a^{-1}_{e(\alpha_1)}$ for all $i$. Moreover,  any other tuple $\{a_i'\}_{i\in Q_0}$, which corresponds to the same $(k_1, \dots, k_{|Q_0|-1})$ differs from $\{a_i\}_{i\in Q_0}$ by some multiple $c\in \kk^*$ such that $a_i'=ca_i$ for $1\leq i\leq |Q_0|-1$. This gives rise to a well-defined map $ (\kk^*)^{|Q_0|-1} \xrightarrow{\theta} D_I$, $(k_1, \dots, k_{|Q_0|-1})\mapsto \overline {f_a}$, where $a=\sum a_ie_i$ (note that $f_a\in H' \cap Inn(A)$).  Clearly $\eta\theta ((k_1, \dots, k_{|Q_0|-1}))=(k_1, \dots, k_{|Q_0|-1})$.

The element $\theta\eta(\overline{f_a})$ is given by $\overline{f_{a'}}$, for some $a'\in A^*$ such that $a-ca'\in J(A)$ for some $c\in \kk^*$. For  $a-ca'\in J(A)$ we have $a(ca)^{-1}=1+r'$ for $r'\in J(A)$, so for any path $p=\beta_{1} \beta_{2}\dots\beta_{k}$ we have $f_af^{-1}_{ca'}(p)=f_{a(ca)^{-1}}(p)=p+r'', r''\in J(A)^{k+1}$. As the order on the basis $P$ agrees with the path length $\overline {f_af^{-1}_{ca'}}=1$ and $\overline {f_a}=\overline {f_{ca'}}=\overline {f_{a'}}$. Thus, $\theta\eta(\overline{f_a})=\overline{f_a}$ and we get the desired bijection.
\end{proof}

Let us consider an algebraic group $D_\Gamma$, which can be constructed from the data $(Q,\pi,m,\mathcal{L})$, here $\Gamma$ stands for the Brauer graph corresponding to the data $(Q,\pi,m,\mathcal{L})$. The set $\mathcal{L}$ is assumed to be empty in case $char(\kk)\neq 2$. The group $D_\Gamma$ is a subgroup of $(\kk^*)^{2|E(\Gamma)|+1}$. The first $2|E(\Gamma)|$ entries $k_\alpha$ are labelled by the arrows of $Q$, the last entry is denoted by $\underline{k}$. The subgroup is given by the following equations: $(k_\alpha)^2=t_\alpha \underline{k}$ for each deformed loop $\alpha$ and
$\prod_{\alpha\in C} k_\alpha^{m(C)}=\underline{k}$ for each $C\in Q_1/\pi$, the  multiplicity of $C$ is denoted by $m(C)$.

\begin{prop}
Let $A$ be a symmetric stably biserial algebra corresponding to the data $(Q,\pi,m,\mathcal{L})$. Then $D_{H'}\simeq D_\Gamma$. 
\end{prop}

\begin{proof}
First assume that the characteristic of the field $\kk$ is arbitrary. Any $f\in H'$ is defined by its action on the arrows of $Q$: $f(\alpha)=k_\alpha \alpha + \sum_{p>\alpha}k_{\alpha,p}p$. We need to show that there exists an automorphism $f\in H'$ with the tuple $(k_\alpha)$ as the coefficients of $\alpha$ in $f(\alpha)=k_\alpha \alpha + \sum_{p>\alpha}k_{\alpha,p}p$ if and only if there exists $\underline{k}\in \kk^*$ such that the equations from the description of the group $D_\Gamma$ are satisfied.

Assume that  for $(k_\alpha)\in (\kk^*)^{2|E(\Gamma)|}$ there exists $\underline{k}\in\kk^*$ such that $(k_\alpha,\underline{k})\in D_\Gamma$, then we can define $f\in H'$ by  $f(\alpha)=k_\alpha \alpha$. This formula clearly gives an automorphism of $A$ in $H'$.

Let us prove that for any $f\in H'$ the tuple of coefficients $(k_\alpha)$ of $\alpha$ in $f(\alpha)$ is an element of $D_\Gamma$ for some $\underline{k}\in\kk^*$. For that we need to better understand which coefficients $k_{\alpha,p}$ can be non-zero. For an arrow $\alpha$, let us denote by  $\bar{C}(\alpha):=C(\alpha)^{m(C(\alpha))}$ the maximal power of the cycle passing through $\alpha$. Let us show that if $k_{\alpha,p}\neq 0$ and $p$ is not a subpath of $\bar{C}(\pi^i(\alpha))$ for any $i$, then $p=\beta^{-1}\bar{C}(\beta)$ for some arrow $\beta$ with $s(\beta)=e(\alpha)$ and $e(\beta)=s(\alpha)$. Assume this is not the case and let us take $p=\beta_1\cdots\beta_t$ with $t$ minimal, since $s(p)=s(\alpha)$, $e(p)=e(\alpha)$ and $p$ is not a subpath of $\bar{C}(\alpha)$, $p\notin \soc (A)$. Since there are at most two maximal paths going through any vertex, $p$ is a subpath of some $\bar{C}(\delta)$ with $\delta\in Q_1, \delta\neq \alpha$, and thus $p$ is unique. Let $\beta$ be the arrow such that $\beta p \in \bar{C}(\beta)$, such  $\beta$ exists and $e(\beta)=s(\alpha)$ but $\alpha\neq \pi(\beta)$. Note that $\beta$ is not a loop with $\pi(\beta)\neq\beta$, otherwise $\alpha=\pi(\beta)$ and $p$ is a subpath of $\bar{C}(\alpha)$. The relation $f(\beta)f(\alpha)=0$ implies that the coefficient before $\beta p$, which contains $k_\beta k_{\alpha,p}$ should be 0. Assume $\beta p\notin \soc A$, then by the minimality of the length of $p$, $k_{\alpha,p}=0$. So $\beta p\in \soc A$ as desired.

Let us now check that for any $f\in H'$ the set $(k_\alpha)$ satisfies the equations $\prod_{\alpha\in C} k_\alpha^{m(C)}=\underline{k}$ for some $\underline{k}\in\kk$. Fix $\alpha\in Q_1$ and let us compute $f(\bar{C}(\alpha))$. It has a summand $\prod_{\alpha'\in C(\alpha)} k_{\alpha'}^{m(C(\alpha))}\bar{C}(\alpha)$, if it has any other summands, then this summands can only appear in one of the following 3 situations: 

1) as a product of the elements of the form $k_{\pi^i(\alpha),p}p$, where $p$ is not a subpath of $\bar{C}(\pi^j(\alpha))$ for any $j$. This situation is possible only in the case of a caterpillar with two simple modules or in the case $|Q_0|=1$, we do not consider any of these cases.  

2) as a product of subpaths of $\bar{C}(\pi^i(\alpha))$ and paths, which are not subpaths of $\bar{C}(\pi^j(\alpha))$ for any $j$, this is possible only in the situation $|Q_0|=1$ and $A$ has a deformed loop, which we also do not consider.  

3) as a product of subpaths of $\bar{C}(\pi^j(\alpha))$, at least one of which comes from $f(\pi^i(\alpha))$ and is not $\pi^i(\alpha)$. Note that all these subpaths are arrows, otherwise the product is zero. Since at least one of the subpaths comes from $f(\pi^i(\alpha))$ and is not $\pi^i(\alpha)$ all of them come from $f(\pi^j(\alpha))$ but are not $\pi^j(\alpha)$, otherwise we are in the situation $|Q_0|=1$ and $A$ has a deformed loop again. Hence every $\pi^i(\alpha)$ has a parallel arrow and $A$ is a caterpillar, which we do not consider. 

So $f(\bar{C}(\alpha))=\prod_{\alpha'\in C(\alpha)} k_{\alpha'}^{m(C(\alpha))} \bar{C}(\alpha)$. Since the relation $\bar{C}(\alpha)=\bar{C}(\beta)$ holds for any $\alpha,\beta\in Q_1$ with $s(\alpha)=s(\beta)$ and the graph $\Gamma$ is connected, we can denote by $\underline{k}$ the product $\prod_{\alpha'\in C(\alpha)} k_{\alpha'}^{m(C(\alpha))}$ for some fixed $\alpha$ and get that $\prod_{\beta\in C} k_\beta^{m(C)}=\underline{k}$ for any $\pi$-cycle $C$. This finishes the proof in the situation, when the number of deformed loops $d$ is zero. Indeed, for any tuple $(k_\alpha)$ coming from $f\in H'$ we have foung $\underline{k}$ such that $(k_\alpha,\underline{k})\in D_\Gamma$. In particular, the proof is finished for $char(\kk)\neq2$.  

From here on we assume $char(\kk)=2$. Let us deal with the  equations $(k_\alpha)^2=t_\alpha \underline{k}$, for each deformed loop $\alpha$. For a deformed loop $\alpha$ let $w_\alpha$ be the path that makes $\alpha w_\alpha$ into a $\pi$-cycle. Let $m$ be the multiplicity of this cycle. Then $f(\alpha)=k_\alpha \alpha+\sum_{p}k_pp$, where $p$ is of the form: $(w_\alpha\alpha)^i,i=1,\dots m-1$, $\alpha(w_\alpha\alpha)^i,i=0,\dots m-1$, $(w_\alpha\alpha)^iw_\alpha,i=0,\dots m-1$, $\alpha(w_\alpha\alpha)^iw_\alpha,i=0,\dots m-1$. Since $\alpha$ is a deformed loop it appears only in socle relations and there are no restrictions on $k_p$ coming from zero relations of the form $\beta\gamma=0$ for $\gamma\neq\pi(\beta)$ in the algebra. 

Let us consider $f(\alpha)^2$. Assume that $m>1$: the coefficient before $\alpha w_\alpha \alpha$ in $f(\alpha)^2$  should be zero, since the socle of $A$ is preserved by any automorphism. This gives $k_\alpha k_{w_\alpha \alpha}+ k_\alpha k_{ \alpha w_\alpha}=0$, so since $k_\alpha\in \kk^*$, $k_{w_\alpha \alpha}+  k_{ \alpha w_\alpha}=0$. Let us assume $k_{(w_\alpha \alpha)^i}+  k_{ (\alpha w_\alpha)^i}=0$, for $i<j<m$ and prove the same for $j$. Let us consider the coefficient of $\alpha(w_\alpha \alpha)^j$ in $f(\alpha)^2$. For any entry of the form $k_{(\alpha w_\alpha )^i }k_{(\alpha w_\alpha )^{j-i}\alpha }$ there is an entry of the form $k_{(\alpha w_\alpha )^{j-i}\alpha }k_{( w_\alpha \alpha )^i}$ and by induction hypothesis they cancel out. The only entries left are $k_\alpha k_{(w_\alpha \alpha)^j}+  k_{ (\alpha w_\alpha)^j} k_\alpha =0$, so we are done. 

Let us now consider the coefficient of $\alpha^2=(\alpha w_\alpha)^m=(w_\alpha \alpha )^m$ in $f(\alpha)^2$. For any entry of the form $k_{( w_\alpha \alpha )^i w_\alpha} k_{( \alpha w_\alpha  )^{m-i-1}\alpha}$ there is an entry $k_{( \alpha w_\alpha  )^{m-i-1}\alpha} k_{( w_\alpha \alpha )^i w_\alpha} $, they cancel out since $char (\kk)=2$, the entries $k_{( w_\alpha \alpha )^i} k_{( w_\alpha \alpha )^{m-i}}$ and $k_{(\alpha w_\alpha  )^i} k_{(\alpha w_\alpha  )^{m-i}}$ cancel out because of the previous paragraph (which we need only for $m$ even, otherwise they cancel out anyway). So we are left with  $f(\alpha)^2=k_\alpha^2\alpha^2$ and $k_\alpha^2=t_\alpha \underline{k}$, as desired, because of the relations $\alpha^2=t_\alpha\bar{C}(\alpha)$ and the fact that $f(\bar{C}(\alpha))=\underline{k}\bar{C}(\alpha)$.
\end{proof}

\begin{lem}
Let $A$ be a stably biserial algebra corresponding to the data $(Q,\pi,m,\mathcal{L})$. The rank of  $D_{H'}$ is  $|Q_1|-|Q_1/\pi|-d+1$, where $d=|\mathcal{L}|$ is the number of deformed loops in $A$.
\end{lem}

\begin{proof} 
Let us construct an epimorphism $j:  D_{H'} \rightarrow  (\kk^*)^{|Q_1|-|Q_1/\pi|-d+1}$ such that the kernel $ker(j)$ is finite. 

Each $\pi$-cycle contains an arrow, which is not a deformed loop. Let us fix one such arrow in each cycle and denote the collection of these arrows by $\mathcal{F}$. Let us label the elements of  $(\kk^*)^{|Q_1|-|Q_1/\pi|-d+1}$ by $x_\alpha$, $\alpha\in Q_1, \alpha\notin \mathcal{F}\cup\mathcal{L}$ and by an additional indeterminant $x$. Define the map $j$ as follows: $j((k_\alpha,\underline{k})):=((x_\alpha,x))$, where $x_\alpha=k_\alpha, \alpha\notin\mathcal{F}\cup\mathcal{L}$, $x=\underline{k}$. The map $j$ is surjective, since for any tuple $(x_\alpha,x)$ we can define  $k_{\gamma}=\sqrt{t_\gamma x}$ for $\gamma\in \mathcal{L}$ and $k_{\beta}=\sqrt[m(C(\beta))]{x/\prod_{\alpha'\in C(\beta), \alpha'\neq \beta, \alpha'\notin\mathcal{L}} x_{\alpha'}^{m(C(\beta))}\prod_{\alpha'\in C(\beta),  \alpha'\in\mathcal{L}}(\sqrt{t_{\alpha'} x})^{m(C(\beta))}}$ for $\beta\in \mathcal{F}$.

Let us compute the kernel of $j$. The tuple $(k_{\alpha},\underline{k})\in ker(j)$ if and only if $\underline{k}=1$, $k_\alpha=1$ for $\alpha\notin \mathcal{F}\cup\mathcal{L}$, $k_\gamma^2=t_\gamma$ for $\gamma\in \mathcal{L}$, $k_\beta^{2m(C(\beta))}=1/\prod_{\alpha'\in C(\beta),  \alpha'\in\mathcal{L}}t_{\alpha'}^{m(C(\beta))}$ for $\beta\in \mathcal{F}$. This clearly defines a finite group.

Passing to the groups of characters, if necessary, and using the equivalence between the category of diagonalizable groups and finitely generated commutative groups \cite[Theorem 12.9]{Milne}, we see that the rank of  $D_{H'}$ is ${|Q_1|-|Q_1/\pi|-d+1}$.
 \end{proof}

\begin{proof}[Proof of Theorem \ref{TheoremRank}]
Since $D_I$ is connected, its image belongs to the maximal torus in $D_{H'}$ and passing to the groups of characters again, the exact sequence $1\rightarrow D_I\rightarrow D_{H'}\rightarrow D_{H'}/D_I\rightarrow 1$, gives that the rank of  $D_{H'}/D_I$ is $|Q_1|-|Q_1/\pi|-d+1 -|Q_0| +1=|Q_0|-|Q_1/\pi|-d+2=|E(\Gamma)|-|V(\Gamma)|-d+2$. 
\end{proof}

\begin{ex}
Let $char(\kk)=2$ and let $A$ and $A_{def}$ be the algebras from Example \ref{ExRef}. Recall that the Brauer graph $\Gamma$ and the quiver $Q$ associated to both $A$ and $A_{def}$ are
\begin{center}
\begin{tikzpicture}
\node (v1) at (-4,1) {1};
\node (v2) at (-1.5,1) {2};

\draw[<-]  (v1) edge[bend right=30] (v2);

\draw[<-]  (v2) edge[bend right=30] (v1);

\node at (-3.1,1.1) {$\alpha$};
\node at (-2.35,0.85) {$\beta$};
\node at (-0.05,1) {$\gamma$};
\node at (-4.5,1) {$Q:$};

\node (v4) at (-11.85,1) {$\Gamma:$};
\node (v3) at (-8.85,1) {$\cdot$};
\node (v4) at (-11.35,1) {$\cdot$};
\draw[-]  (v3) edge (v4);
\draw[-]  plot[smooth, tension=.7] coordinates {(-8.7,0.85) (-7.85,0.7) (-7.85,1.3) (-8.7,1.15)};

\draw[<-]  plot[smooth, tension=.7] coordinates {(-1.2,0.85) (-0.35,0.7) (-0.35,1.3) (-1.2,1.15)};

\end{tikzpicture}.
\end{center}

\noindent $A$ is a Brauer graph algebra and $A_{def}$ is a symmetric stably biserial algebra with $1$ deformed loop. Thus $A$ corresponds to the data $(Q,\pi,m,\emptyset)$ and $A_{def}$ corresponds to the data $(Q,\pi,m,\{\gamma\})$ with $|\{\gamma\}|=1$, where the permutation $\pi$ is constructed from the Brauer graph $\Gamma$ for both algebras. 
The rank of the maximal torus $T(A)$ of $Out^0(A)$ is $|E(\Gamma)|-|V(\Gamma)|-0+2=2-2-0+2=2$ and the rank of the maximal torus $T(A_{def})$ of $Out^0(A_{def})$ is $|E(\Gamma)|-|V(\Gamma)|-1+2=2-2-1+2=1$. Algebras $A$ and $A_{def}$ are not derived equivalent.

If $char(\kk)\neq 2$ we have already seen that $A\simeq A_{def}$ and we need to use the description of $A$ in order to compute the rank. In that case the rank of $T(A)$ is $2$.
\end{ex}

Using the fact that Brauer graph algebras can be stably (and hence derived) equivalent only to symmetric stably biserial algebras \cite[Theorem 1 and 3]{AZ1}, Corollary \ref{CorClosed} can be deduced from Theorems \ref{TheoremRank} and \ref{TheoremDerInv} and Corollary \ref{CatLem}, as well as from the fact that for local algebras derived equivalence implies Morita equivalence.


\begin{thebibliography}{99}
\bibitem{Gnedin} Gnedin, Wassilij. "Ribbon graph orders are preserved by derived equivalences." In preparation.
\end{thebibliography}
\end{document}